\documentclass[11pt, a4paper, oneside,leqno]{article}
\usepackage{amsmath, amsfonts, amsthm,dsfont,enumerate, amssymb, fouriernc}
\usepackage{authblk}

\usepackage{color}

\DeclareMathOperator{\sinc}{sinc}

\newtheorem{mylem}{Lemma}[section]
\newtheorem{mythe}{Theorem}[section]
\newtheorem{myrem}{Remark}[section]
\newtheorem{cor}{Corollary}[section]
\newtheorem{prop}{Proposition}[section]

\newcommand{\ep}{\varepsilon}

\catcode`,\active

\catcode`\,12

\title{Remarkable properties of the $\sinc_{p, q}$ functions and related integrals}

\author[1]{Houry Melkonian}
\author[2]{Shingo Takeuchi}
\affil[1]{{\small \it Department of Mathematics, College of Engineering, Mathematics and Physical Sciences, University of Exeter, Penryn TR10 9FE, UK.}}
\affil[2]{{\small \it Department of Mathematical Sciences, College of Systems Engineering and Science, Shibaura Institute of Technology, Saitama 337-8570, Japan.}}

\begin{document}
\maketitle
\footnotetext[1]{Email address: \texttt{h.melkonian@exeter.ac.uk}}
\footnotetext[2]{Email address: \texttt{shingo@shibaura-it.ac.jp}}

\begin{abstract}
Different types of sinc integrals are investigated when the standard sine function is replaced by the generalised $\sin_{p,q}$ in two parameters.
A striking generalisation of the improper Dirichlet integral is achieved. A second surprising but interesting generalisation of the identity between the Dirichlet integral and that of the integrand $\sinc^2$ is discovered. Moreover, an asymptotically sharpened form of Ball's integral inequality is obtained in terms of two parameters.
\end{abstract}

The main objective of this paper is to establish an understanding of the two integrals
\begin{align}\label{2int}
\int_0^\infty \left|\frac{\sin_{p,q} x}{x}\right|^{r-1}  \frac{\sin_{p,q} x}{x}dx \qquad \text{and} \qquad  \int_0^\infty \left\vert\frac{\sin_{p,q} x}{x}\right\vert^m dx
\end{align}
when $r>0 $ and $m, p, q \in (1, \infty)$. The functions $\sin_{p,q}$ in the 
integrands
are deformations of the classical trigonometric sine when controlled by the values of the parameters $p, q \in (1, \infty)$ and are known as the generalised trigonometric sine functions. They appear as eigenfunctions of 
the nonhomogeneous eigenvalue problem with the $p$-Laplacian:
\begin{equation}
\label{eq:ep}
-(|u'|^{p-2}u')'=\lambda |u|^{q-2}u
\end{equation}
under a Dirichlet boundary condition 
and are defined as the inverse of the function $F_{p,q}: \left[0,1\right]\rightarrow \left[0, \frac{\pi_{p,q}}{2}\right]$ which is given by the Abelian integral:
\[
\sin_{p,q}^{-1} x:=F_{p,q}(x)=\int_0^x (1-t^q)^{-1/p} dt, \qquad x \in [0,1],
\]
where $\pi_{p,q}/2:=F_{p,q}(1)=B(1/p^{*},1/q)/q$ ($B$ is the beta function and $p^*:=p/(p-1)$). The generalised trigonometric sine is an increasing function on the interval $ \left[0, \frac{\pi_{p,q}}{2}\right]$ and is extended to the whole real line as a $2\pi_{p,q}$-periodic function by means of $\sin_{p,q}x=\sin_{p,q}(\pi_{p,q}-x)$ and $\sin_{p,q} (-x)=-\sin_{p,q} x$. Various properties of these functions and their counterparts $\cos_{p,q}x:=(\sin_{p,q} x)'$ and $\tan_{p,q}x:=(\sin_{p,q} x)/(\cos_{p,q} x)$ are explored. See \cite{ baksi2018basis,  BBCDG2006,  BL2010, BM2016, EdmundsGurkaLang2012, EdmundsLang2011, lindqvist1993note, 1984Otani, takeuchi2016multiple}.

The research we present in our paper is new of its kind, it started when D. E. Edmunds proposed the question about whether it is possible to form 
generalisations of the well-known integral inequality by K. Ball \cite{ball1986cube}:
\begin{equation}
\label{eq:ball}
\sqrt{m} \int_{-\infty}^\infty \left\vert\frac{\sin x}{x}\right\vert^m dx \leq \sqrt{2}\pi \qquad \text{for} \quad m \geq2,
\end{equation}
and of its asymptotically sharpened form (as in \cite{borwein2010p} and \cite{kerman2015asymptotically})
\begin{align}\label{BII}
\sqrt{m} \int_0^\infty \left\vert\frac{\sin x}{x}\right\vert^m dx \sim \sqrt{\frac{3\pi}{2}}\left(1-\frac{3}{20}  \frac{1}{m}-\frac{13}{1120}\frac{1}{m^2}\right)+\sum_{j=3}^\infty \frac{c_j}{m^j} \quad \text{as} \quad m \to \infty,
\end{align}
in the case when the standard trigonometric sine function is replaced by its generalisation $\sin_p:=\sin_{p,p}$ for $p>1$. The latter equivalence was explored in the works of D. Borwein, J. M. Borwein, I. E. Leonard, R. Kerman, R. Ol'hava, S. Spektor and many others \cite{ball1986cube, borwein2001some, borwein2010p, kerman2015asymptotically} about integrals of standard sinc function. 

Inspired by these, a first attempt to answer Edmund's question was investigated in the paper \cite{edmunds2019behaviour}, and results about the asymptotic nature of the integral on the left-hand side of Ball's inequality were obtained. 

In the current paper we aim to provide further developments to the existing research. Our first result is Theorem \ref{dirintpq}, which justifies an identity involving integrals of integer powers of the $(\sin_{p,q} x)/x$. A second important outcome of the paper is expressed in Theorem \ref{sinqq'}, where an integral identity between two kinds of generalised trigonometric functions, $\sin_{2,q}$ and $\sin_{q^*,q}$, is achieved. Both theorems provide fruitful generalisations of the classical identities:
\begin{align}\label{dirint22}
\int_0^\infty \left(\frac{\sin x}{x}\right)^2 dx=\int_0^\infty \frac{\sin x}{x} dx=\frac{\pi}{2},
\end{align}
which are contained in the theorems as the case $p=q=2$.

In addition, Theorem \ref{asyexp} and Corollary \ref{pqBall} communicate results about the second integral in \eqref{2int} providing a generalisation of the asymptotically sharpened form \eqref{BII} of the left-hand side of Ball's integral inequality \eqref{eq:ball}.    \\                   

For $p, q\in (1,\infty ),$ define the function sinc$_{p, q}$ to be $\sinc_{p, q}x:=(\sin_{p,q}x)/x$ when $ x\in \mathbb{R}\backslash \{0\}$ and $\sinc_{p,q} 0:=1$.
It is obvious that $\sinc_{p,q}$ is an even function and has roots at $n\pi _{p, q}$ when $n\in \mathbb{Z}%
\backslash \{0\}.$ Moreover, since $\left\vert \sin _{p,q}x\right\vert \leq 1$ for all $%
x\in \mathbb{R},$ $\lim_{\left\vert x\right\vert \rightarrow \infty }$sinc$%
_{p, q}x=0.$

\begin{mylem}\label{sincpmonotonic}
\begin{enumerate}[(i)]
\item \label{1a} For all $x \in \mathbb{R}$, $\vert \sinc_{p, q} x \vert\leq1$.
\item \label{1b} The function $\sinc_{p, q} $ is strictly decreasing on the interval $\left(0, \frac{\pi_{p, q}}{2}\right]$.
\end{enumerate}
\end{mylem}
\begin{proof}
The proof is analogous to that of \cite[Lemma 2.1]{edmunds2019behaviour}, hence omitted.
\end{proof}
\section{Improper Riemann integrals of $\sinc_{p,q}$ functions}
It is known that the improper integrals \eqref{dirint22} of the $\sinc$ function over the positive interval exist in the sense of Riemann integrals.
We start this section by a precise consideration of the convergence properties of the integrals in \eqref{2int}.
\begin{prop}
Let $p,\ q \in (1,\infty)$. Then,
\begin{enumerate}[(i)]
\item $\displaystyle \int_0^\infty \left|\frac{\sin_{p,q}{x}}{x}\right|^{r-1}\frac{\sin_{p,q}{x}}{x}\,dx$ 
is convergent for $r>0$.
\item $\displaystyle \int_0^\infty \left| \frac{\sin_{p,q}{x}}{x} \right|^m\,dx$ 
is convergent for $m>1$ and divergent for $m=1$.
\end{enumerate}
\end{prop}

\begin{proof}
Since $\sinc_{p,q}{x}$ is continuous in $[0,\infty)$, it suffices to observe the
integrability for $x$ sufficiently large. 

(i) Let $f(x):=|\sin_{p,q}{x}|^{r-1}\sin_{p,q}{x}$ and 
$$F(x):=\int_0^x f(t)\,dt, \quad x \geq 0.$$
Then, $F(x)$ is continuous in $[0,\infty)$. 
Since $f(x)$ is $2\pi_{p,q}$-periodic, we see that $F(x+2\pi_{p,q})-F(x)=F(2\pi_{p,q})=0$.
Thus, $F(x)$ is also $2\pi_{p,q}$-periodic.
Hence, $F(x)$ is bounded on $[0,\infty)$
and there exists $M>0$ such that $|F(x)| \leq M$ for all $x \in [0,\infty)$.
Therefore, when $0<u<v$,
\begin{align*}
\left| \int_u^v \frac{f(x)}{x^r}\,dx \right|
&=\left| \left[ \frac{F(x)}{x^r} \right]_u^v+r\int_u^v \frac{F(x)}{x^{r+1}}\,dx \right|\leq \frac{M}{v^r}+\frac{M}{u^r}-\frac{M}{v^r}+\frac{M}{u^r}=\frac{2M}{u^r}.
\end{align*}
As $u \to \infty$, Cauchy's test yields the convergence of the improper integral.

(ii) For $m>1$, since $|\sinc_{p,q}{x}|^m<1/x^m$, the convergence is trivial.

Consider the case $m=1$. For $n=0,1,2,\ldots$,
\begin{align*}
\int_{n \pi_{p,q}}^{(n+1)\pi_{p,q}} \left| \frac{\sin_{p,q}{x}}{x} \right| \,dx
&=\int_0^{\pi_{p,q}} \frac{\sin_{p,q}{x}}{n\pi_{p,q}+x} \,dx\\
&>\frac{1}{(n+1) \pi_{p,q}} \int_0^{\pi_{p,q}} \sin_{p,q}{x}\,dx\\
&=:\frac{C}{n+1},
\end{align*}
where $C=2B\left(2/q,1/p^*\right)/(q\pi_{p,q})$ by \cite[Theorem 3.1]{Takeuchi2019}.
Therefore,
$$\int_0^{n\pi_{p,q}} \left| \frac{\sin_{p,q}{x}}{x} \right| \,dx
>C\sum_{k=1}^n \frac{1}{k}.$$
The last summation diverges to $\infty$ as $n \to \infty$, and so does the left-hand side.
\end{proof}
Now we give analogues of the integrals in \eqref{dirint22}
for the set of non-standard trigonometric functions.
\begin{mythe}\label{dirintpq}
For $p,\ q\in (1,\infty)$, 
\begin{gather}
\int_0^\infty \frac{\sin_{p,q}{x}}{x}\,dx
=\frac{\pi}{2} \int_0^1 \frac{\sin_{p,q}{(\pi_{p,q}x)}}{\sin{(\pi x)}}\,dx, \label{eq:L1}\\
\int_0^\infty \left(\frac{\sin_{p,q}{x}}{x}\right)^2\,dx
=\frac{\pi^2}{2\pi_{p,q}} \int_0^1 \left[\frac{\sin_{p,q}{(\pi_{p,q}x)}}{\sin{(\pi x)}}\right]^2\,dx.
\label{eq:L2}
\end{gather}
In particular, when $p=q=2$,
\begin{equation}
\label{eq:si}
\int_0^\infty \frac{\sin{x}}{x}\,dx=\frac{\pi}{2}
=\int_0^\infty \left(\frac{\sin{x}}{x}\right)^2\,dx.
\end{equation}
\end{mythe}

\begin{proof}
Let $r \in \mathbb{N}$.
Since $\sinc_{p,q}{x}$ is an even function, 
\begin{align*}
\int_0^\infty \left(\frac{\sin_{p,q}{x}}{x}\right)^r\,dx
&=\frac12 \int_{-\infty}^{\infty} \left(\frac{\sin_{p,q}{x}}{x}\right)^r\,dx\\
&=\frac12 \sum_{n=-\infty}^{\infty} \int_{n\pi_{p,q}}^{(n+1)\pi_{p,q}}
\left(\frac{\sin_{p,q}{x}}{x}\right)^r\,dx.
\end{align*}
Setting $x=(t+n)\pi_{p,q}$, we have
\begin{equation}\label{eq:int}
\int_0^\infty \left(\frac{\sin_{p,q}{x}}{x}\right)^r\,dx
=\frac{1}{2\pi_{p,q}^{r-1}} \sum_{n=-\infty}^{\infty} \int_0^1 \sin_{p,q}^r{(\pi_{p,q}t)}
\frac{(-1)^{rn}}{(t+n)^r}\,dt.
\end{equation}


To guarantee the interchange between the infinite sum and the integral, it suffices to show that 
\begin{equation}
\label{eq:uc}
L_r(t):=\sum_{n=-\infty}^\infty \frac{(-1)^{rn}}{(t+n)^r}
\end{equation}
converges uniformly on $(0,1)$.
For any $N \in \mathbb{N}$ and $t \in (0,1)$, we define
$$S_{r,N}(t):=\sum_{n=-N}^N \frac{(-1)^{rn}}{(t+n)^r}.$$
For $N \ge 2$, we have
\begin{align*}
S_{r,N}(t)
&=\frac{1}{t^r}
+\frac{(-1)^r((t-1)^r+(t+1)^r)}{(t^2-1)^r}
+\sum_{n=2}^N \frac{(-1)^{rn}((t-n)^r+(t+n)^r)}{(t^2-n^2)^r}.
\end{align*}
For $t \in (0,1)$ and $n \geq  2$, clearly $t < n/2$. Then, 
it is easy to see that $|(t-n)^r+(t+n)^r| \leq c_1 n^{r-1}$ for some $c_1>0$
when $r$ is odd; and $|(t-n)^r+(t+n)^r| \leq c_2 n^r$ for some $c_2>0$
when $r$ is even. Similarly, $|(t^2-n^2)^r|>c_3n^{2r}$ for some $c_3>0$.
Thus, for $r \in \mathbb{N}$, letting $c:=\max\{c_1, c_2\}/c_3$ gives
$$\left| \frac{(-1)^{rn}((t-n)^r+(t+n)^r)}{(t^2-n^2)^r} \right|
< \frac{c}{n^2}$$
for $t \in (0,1)$ and $n \geq 2$.
Since $\sum_{n=2}^\infty (c/n^2)<\infty$, we conclude that
$S_{r,N}(t)$ converges uniformly to $L_r(t)$ on $(0,1)$ as $N \to \infty$.

By \eqref{eq:int} and \eqref{eq:uc}, we obtain
\begin{align}\label{Lr}
\int_0^\infty \left(\frac{\sin_{p,q}{x}}{x}\right)^r\,dx
=\frac{1}{2\pi_{p,q}^{r-1}} 
\int_0^1 \sin_{p,q}^r{(\pi_{p,q}t)} L_r(t) \,dt.
\end{align}
From \cite[1.422, p.44]{Grad2007}, 
$$L_1(t)=\frac{\pi}{\sin{(\pi t)}}, \quad L_2(t)=\frac{\pi^2}{\sin^2{(\pi t)}};$$
hence \eqref{eq:L1} and \eqref{eq:L2} hold.
\end{proof}

Regarding Theorem \ref{dirintpq}, 
it is worth pointing out that \eqref{eq:L1} and \eqref{eq:L2} contain \eqref{eq:si}. 
It follows from \cite[Lemma 4.2]{Takeuchi2014} that, the function $\sin_{p,q}{(\pi_{p,q}{x})}$ (with $x \in (0,1/2) \cup (1/2,1)$) is strictly 
decreasing in any of its parameters $p$ or $q \in (1, \infty)$ whenever the second is fixed. Therefore, the right-hand sides of \eqref{eq:L1} and \eqref{eq:L2}
are strictly decreasing functions in the same parameter. 
From this fact, we see that (letting $p=q$ for simplicity)
$$\int_0^\infty \frac{\sin_p{x}}{x}\,dx
\gtreqqless \frac{\pi}{2} \quad \mbox{and}\quad
\int_0^\infty \left(\frac{\sin_p{x}}{x}\right)^2\,dx
\gtreqqless \frac{\pi^2}{2\pi_p} \quad \mbox{if}\ p \lesseqqgtr 2.$$

It is also possible to obtain the formulas of integrals for $r \geq 3$.
By the summation theorems \cite[Theorem 4.4.1, p.305 \& Exercise 5, p.313]{marsden1999basic}, which follow from the residue theorem, we have
\begin{equation}
\label{eq:L_r}
L_r(t)=
-\frac{1}{(r-1)!}\lim_{z \to -t} \frac{d^{r-1}}{dz^{r-1}}\left[\frac{\pi}{\varphi(\pi z)}\right],
\end{equation}
where $\varphi(z)=\sin{z}$ if $r$ is odd; and $\varphi(z)=\tan{z}$ if $r$ is even. For instance, \eqref{eq:L_r} gives
$$L_3(t)=\frac{\pi^3(2-\sin^2{(\pi t)})}{2\sin^3{(\pi t)}},
\quad L_4(t)=\frac{\pi^4(3-2\sin^2{(\pi t)})}{3\sin^4{(\pi t)}};$$
hence by \eqref{Lr} we have,
\begin{gather*}
\int_0^\infty \left(\frac{\sin_{p,q}{x}}{x}\right)^3\,dx
=\frac{\pi^3}{4\pi_{p,q}^2} \int_0^1 \left[\frac{\sin_{p,q}{(\pi_{p,q}x)}}{\sin{(\pi x)}}\right]^3
(2-\sin^2{(\pi x)})\,dx,\\
\int_0^\infty \left(\frac{\sin_{p,q}{x}}{x}\right)^4\,dx
=\frac{\pi^4}{6\pi_{p,q}^3} \int_0^1 \left[\frac{\sin_{p,q}{(\pi_{p,q}x)}}{\sin{(\pi x)}}\right]^4
(3-2\sin^2{(\pi x)})\,dx.
\end{gather*}
These contain
$$\int_0^\infty \left(\frac{\sin{x}}{x}\right)^3\,dx=\frac{3\pi}{8},\quad
\int_0^\infty \left(\frac{\sin{x}}{x}\right)^4\,dx=\frac{\pi}{3},$$
respectively.

\begin{myrem}\label{remark1}
The values of the integrals \eqref{eq:L1} and \eqref{eq:L2} for $(p, q) \neq (2, 2)$
still remain as open questions.
\end{myrem}

Theorem \ref{dirintpq} shows a relationship between the integrals
of the left-hand sides of \eqref{eq:L1} and \eqref{eq:L2} as follows.

\begin{cor}
For $p,\ q \in (1,\infty)$,
\begin{equation}
\label{eq:schwarz}
\left( \int_0^\infty \frac{\sin_{p,q}{x}}{x}\,dx \right)^2
\leq \frac{\pi_{p,q}}{2} \int_0^\infty \left( \frac{\sin_{p,q}{x}}{x} \right)^2\,dx.
\end{equation}
Moreover, the equality holds if and only if $p=q=2$.
\end{cor}

\begin{proof}
Applying Schwarz's inequality to \eqref{eq:L1}, we obtain
$$\left(\int_0^\infty \frac{\sin_{p,q}{x}}{x}\,dx \right)^2
\leq \frac{\pi^2}{4}
\int_0^1 \left[\frac{\sin_{p,q}{(\pi_{p,q}x)}}{\sin{(\pi x)}}\right]^2\,dx.$$
By \eqref{eq:L2}, the right-hand side is equal to
$$\frac{\pi_{p,q}}{2} \int_0^\infty \left( \frac{\sin_{p,q}{x}}{x} \right)^2\,dx.$$

The equality of \eqref{eq:schwarz}
holds if and only if there exists a constant $k$ such that 
$\sin_{p,q}{(\pi_{p,q}x)}=k \sin{(\pi x)}$ for all $x \in [0,1]$. 
Then, as $x=1/2$, we see that $k=1$.
As described immediately after the proof of Theorem \ref{dirintpq},
$\sin_{p,q}{(\pi_{p,q}x)}$, $x \in (0,1/2) \cup (1/2,1)$,
is strictly decreasing in $p,\, q$. Therefore, we conclude $p=q=2$.
\end{proof}

The paper \cite{takeuchi2016multiple} states the so-called multiple-angle formula between two types of generalised trigonometric functions: for all $x \in \mathbb{R}$ and $q \in (1,\infty)$,
\begin{align}\label{eq:maf}
\sin_{2,q}{(2^{2/q}x)}=2^{2/q}\sin_{q^*,q}{x}|\cos_{q^*,q}{x}|^{q^*-2}\cos_{q^*,q}{x}.
\end{align}

This formula is an essential tool to establish the following result.

\begin{mythe}\label{sinqq'}
For $q\in (1,\infty)$, 
\begin{align}\label{the21}
\int_0^{\infty} \left|\frac{\sin_{q^*,q}{x}}{x}\right|^{q}\,dx=\frac{q^*}{2^{2/q}}\int_0^{\infty} \left|\frac{\sin_{2,q}{x}}{x}\right|^{q-2}\frac{\sin_{2,q}{x}}{x}\,dx.
\end{align}
Hence, when $q$ is even, 
$$\int_0^{\infty} \left(\frac{\sin_{q^*,q}{x}}{x}\right)^{q}\,dx=\frac{q^*}{2^{2/q}}\int_0^{\infty} \left(\frac{\sin_{2,q}{x}}{x}\right)^{q-1}\,dx;$$
in particular, when $q=2$,
$$\int_0^{\infty} \left(\frac{\sin{x}}{x}\right)^2\,dx=\int_0^{\infty} \frac{\sin{x}}{x}\,dx.$$
\end{mythe}

\begin{proof}
From \eqref{eq:maf},
\begin{align*}
|\sin_{2,q}{(2^{2/q}x)}|^{q-2}\sin_{2,q}{(2^{2/q}x)}
=2^{2/q^*}|\sin_{q^*,q}{x}|^{q-2}\sin_{q^*,q}{x} \cos_{q^*,q}{x}.
\end{align*}
Using this identity, we obtain
\begin{align*}
\int_0^\infty \left|\frac{\sin_{q^*,q}{x}}{x}\right|^q\,dx
&=\left[\frac{1}{1-q}x^{1-q}|\sin_{q^*,q}{x}|^q\right]_0^\infty\\
& \qquad -\frac{1}{1-q}\int_0^\infty x^{1-q}\cdot q
|\sin_{q^*,q}{x}|^{q-2}\sin_{q^*,q}{x} \cos_{q^*,q}{x}\,dx\\
&=q^*\int_0^\infty
x^{1-q}\cdot 2^{-2/q^*}
|\sin_{2,q}{(2^{2/q}x)}|^{q-2}\sin_{2,q}{(2^{2/q}x)}\,dx\\
&=\frac{q^*}{2^{2/q}}
\int_0^{\infty} \left|\frac{\sin_{2,q}{x}}{x}\right|^{q-2}\frac{\sin_{2,q}{x}}{x}\,dx.
\end{align*}
Therefore, the proof is complete. 
\end{proof}

It is known that $\sin_{p,q}{x}$ 
satisfies \eqref{eq:ep} with $\lambda=q/p^*$, i.e.
$$-(|\cos_{p,q}{x}|^{p-2}\cos_{p,q}{x})'=\frac{q}{p^*}|\sin_{p,q}{x}|^{q-2}\sin_{p,q}{x}.$$
Integration by parts yields,
\begin{align}\label{rem2}
\int_0^\infty \frac{1-|\cos_{p,q}{x}|^{p-2}\cos_{p,q}{x}}{x^q}\,dx
=\frac{q^*}{p^*}\int_0^\infty \left| \frac{\sin_{p,q}{x}}{x} \right|^{q-2}
\frac{\sin_{p,q}{x}}{x}\,dx.
\end{align}
Now, from \eqref{rem2}, both sides of \eqref{the21} are equal to 
\begin{align*}
2^{1-2/q} \int_0^\infty \frac{1-\cos_{2,q}{x}}{x^q}dx.
\end{align*}
In particular, when $q=2$,
the famous equalities are obtained:
\[
\int_0^\infty \left(\frac{\sin{x}}{x}\right)^2\,dx=\int_0^\infty \frac{\sin{x}}{x}\,dx=\int_0^\infty \frac{1-\cos{x}}{x^2}\,dx.
\]

\section{The $L_m$-norm behaviour of $\sinc_{p,q}$ functions}\label{sec2}

Here in Theorem \ref{asyexp} below we pay a close attention to the $L_m$-norm behaviour of the $\sinc_{p,q}$ function and its asymptotic expansion (with explicit first two terms) for large values of $m$. Independent of that, we calculate the limit of this integral as $m \to \infty$. But before we embark on this study, we highlight the following lemma, which is essentially due to L.I. Paredes and K. Uchiyama \cite[Theorem 3.2]{paredes2003analytic}.

\begin{mylem}\label{convsin}
Let $p, q \in (1, \infty)$. Then, the function $\sin_{p,q}x$ has the convergent expansion near $x=0$:
\begin{align*}\label{sinpower}
\sin_{p,q} x=x-\frac{1}{p(q+1)}|x|^{q}x+\frac{1-p+3q-pq}{2p^2(q+1)(2q+1)}|x|^{2q}x+\cdots.
\end{align*}
\end{mylem}
\begin{proof}
The function $u(t)=\sin_{p,q}(kt)$ satisfies 
\eqref{eq:ep} with $\lambda=1$
on $\left(-\frac{\pi_{p,q}}{k},\frac{\pi_{p,q}}{k}\right)$ when $k^p=p^*/q$. Therefore, 
applying \cite[Theorem 3.2]{paredes2003analytic} to our case,
with $\sigma=0$ and $A=k$, we see that $u(t)$ has the convergent expansion near $t=0$:
$$u(t)=kt-\frac{k^{q-p+1}}{(p-1)q(q+1)}|t|^qt+\frac{(1-p+3q-pq)k^{2q-2p+1}}{2(p-1)^2q^2(q+1)(2q+1)}|t|^{2q}t+\cdots.$$
Setting $x=kt$ gives the expansion in the lemma.
\end{proof}
Let us now explore the behaviour of the $L_m$-norm of $\sinc_{p,q}$ when $m$ is large enough. 

For $m, p, q \in (1, \infty)$, define
\[
I_{p, q}(m):= m^{1/q}\int_0^\infty \left\vert\frac{\sin_{p, q} x}{x}\right\vert^m dx.
\]
Then we can show the following lemma.
\begin{mylem}
\label{lem:Ipqm}
Let $p,\ q \in (1,\infty)$. Then, for any $\alpha \in (0,\infty)$, 
\begin{align*}\label{Ipqm}
\lim_{m \to \infty}I_{p, q}(m)= \lim_{m \to \infty}m^{1/q}\int_0^\alpha \left\vert\frac{\sin_{p, q} x}{x}\right\vert^m dx.
\end{align*}
Hence, the value of the limit is independent of $\alpha$.
\end{mylem}
\begin{proof}
For $\alpha \in (0, \infty)$, we define $I_{p,q}(m)$ as
\[
I_{p, q}(m)= m^{1/q}\int_0^\alpha \left\vert\frac{\sin_{p, q} x}{x}\right\vert^m dx+m^{1/q}\int_\alpha^\infty \left\vert\frac{\sin_{p, q} x}{x}\right\vert^mdx.
\]

Let $\alpha \in [1, \infty)$. In this case,
\begin{align*}
 \int_\alpha ^\infty \left\vert \frac{\sin_{p,q} x}{x} \right\vert^mdx &= \displaystyle\lim_{\beta \rightarrow \infty}  \int_\alpha ^\beta \left\vert \frac{\sin_{p, q} x}{x} \right\vert^m dx\leq \displaystyle\lim_{\beta \rightarrow \infty} \int_\alpha ^\beta x^{-m} \mathrm{d}x=\frac{1}{m-1}\alpha^{1-m}.
\end{align*}
Then,
\[
\displaystyle\lim_{m \rightarrow \infty} m^{1/q} \int_\alpha ^\infty \left\vert \frac{\sin_{p, q} x}{x} \right\vert^m dx \leq \displaystyle\lim_{m \rightarrow \infty} \frac{m^{1/q}}{(m-1)\alpha^{m-1}}=0.
\]

Next, when $\alpha \in(0,1)$,
\begin{align*}
\int_\alpha ^\infty \left\vert \frac{\sin_{p, q} x}{x} \right\vert^m dx &= \int_\alpha ^1 \left\vert \frac{\sin_{p, q} x}{x} \right\vert^m dx + \int_1 ^\infty \left\vert \frac{\sin_{p, q} x}{x} \right\vert^m dx\\
&\leq \int_\alpha ^1 \left\vert \frac{\sin_{p, q} x}{x} \right\vert^m dx + \frac{1}{m-1}.
\end{align*}
From \cite[(2.17), p.39]{EdmundsLang2011}, observe that $(\alpha, 1) \subset (0, \frac{\pi_{p, q}}{2})$. Then by Lemma \ref{sincpmonotonic} \eqref{1b} we have
\[
0<\frac{\sin_{p, q} x}{x}<\frac{\sin_{p, q} \alpha}{\alpha}<1, \qquad x \in (\alpha, 1).
\]
Then, 
\[
\displaystyle\lim_{m \rightarrow \infty} m^{1/q} \int_\alpha ^\infty \left\vert \frac{\sin_{p, q} x}{x} \right\vert^m dx  \leq \displaystyle\lim_{m \rightarrow \infty} m^{1/q}\left[\left(\frac{\sin_{p, q} \alpha}{\alpha}\right)^m(1-\alpha)+\frac{1}{m-1}\right]=0,
\]
and the lemma follows.
\end{proof}

With the aid of Lemma \ref{lem:Ipqm}, we now present Theorem \ref{asyexp} and Corollary \ref{pqBall} below. 
Because of the independency of $\alpha$ in Lemma \ref{lem:Ipqm}, it suffices to complete the study on a positive interval $(0,\alpha)$ for some $\alpha>0$.
In what follows we use the notation
$f(x) \sim g(x)$ for $f(x)/g(x) \to 1$ as $x \to \infty$.
\begin{mythe}\label{asyexp}
Let $p, q\in(1,\infty)$. Then, there exist constants $\gamma_2,\ \gamma_3,\ \ldots$, independent of $m$, such that 
for $m$ large enough
\begin{multline*}
I_{p,q}(m) \sim \frac{1}{q} \Gamma \left(\frac{1}{q}\right) (p(q+1))^{1/q}
\left[ 1-\frac{(q+1)(pq^2+2pq-3q^2+p-2q)}{2q^2(2q+1)}\frac{1}{m}\right]\\
+\frac{1}{q}\sum_{j=2}^\infty \Gamma\left(j+\frac{1}{q}\right)\frac{\gamma_j}{m^j}.
\end{multline*}
\end{mythe}

\begin{proof}
Let $\alpha \in (0, 1)$. Then, by Lemma \ref{lem:Ipqm},
\begin{align}\label{Jm}
I_{p,q}(m) \sim J(m,\alpha):=m^{1/q}\int_0^{\alpha} 
\left(\frac{\sin_{p,q}{x}}{x}\right)^m\,dx.
\end{align}

In what follows, we observe the asymptotic expansion of $J(m,\alpha)$
instead of $I_{p,q}(m)$.
To use Theorem 8.1 in Olver's book \cite[p.86]{olver1974asymptotics}, we rewrite $J(m,\alpha)$ as
\begin{align*}
J(m,\alpha)
&=m^{1/q}\int_0^\alpha \exp{\left( m \ln{\left(\frac{\sin_{p,q}{x}}{x}\right)}\right)}\,dx\\
&=\frac{m^{1/q}}{q} \int_0^{\alpha^q} e^{-mf(t)}g(t)\,dt,
\end{align*}
where
\begin{gather*}
f(t):=-\ln{ \left( \frac{\sin_{p,q}{(t^{1/q})}}{t^{1/q}} \right) },\quad
g(t):=t^{1/q-1}.
\end{gather*}
It is easy to see that $f(t)$ and $g(t)$ satisfy (i),\ (ii), and (iv) in \cite[\S{7.2}]{olver1974asymptotics}.
We need to find 
constants $\mu,\ f_j,\ \lambda$ and $g_j\ (j=0,1,2,\ldots)$ such that $\mu, \lambda>0, \ f_0, g_0 \neq0$ and
as $t \to 0^+$,  
\begin{gather*}
f(t) \sim f(0)+\sum_{j=0}^\infty f_j t^{j+\mu},\quad
g(t) \sim \sum_{j=0}^\infty g_j t^{j+\lambda-1}.
\end{gather*}
It is clear that
\begin{equation}
\label{eq:g_j}
\lambda=\frac{1}{q},\quad g_0=1,\quad g_j=0\ (j \geq 1).
\end{equation}
Regarding $\mu$ and $f_j$, we will expand $f(t)$.
By Lemma \ref{convsin}, taking $\alpha$ small if necessary, we obtain  
$$\frac{\sin_{p,q}{(t^{1/q})}}{t^{1/q}}=\sum_{j=0}^\infty a_j t^j, \quad t \in (0,\alpha^q),$$
where the first three terms of $a_j$ are
\begin{equation}
\label{eq:a_j}
a_0=1,\quad a_1=-\frac{1}{p(q+1)},\quad a_2=\frac{1-p+3q-pq}{2p^2(q+1)(2q+1)}.
\end{equation}
Since
\[
\left\vert \sum_{j=1}^\infty a_j t^j \right\vert =\left\vert 
\frac{\sin_{p,q}{(t^{1/q})}}{t^{1/q}}-1\right\vert<1, \quad t \in(0, \alpha^q),
\]
it is immediate that for $t \in(0, \alpha^q)$,
\begin{align*}
f(t)
=-\ln{\left(1+\sum_{j=1}^\infty a_j t^j\right)}
=\sum_{k=1}^\infty \frac{(-1)^k}{k} \left(\sum_{j=1}^\infty a_j t^j\right)^k,
\end{align*}
which yields
\begin{align*}
f(t)=-a_1t+\left(-a_2+\frac{a_1^2}{2}\right)t^2+O(t^3).
\end{align*}
This means that $\mu$ and the first two terms of $f_j$ are
\begin{equation}
\label{eq:f_j}
\mu=1,\quad f_0=-a_1,\quad f_1=-a_2+\frac{a_1^2}{2}.
\end{equation}

We are now in a position to give the asymptotic expansion of $J(m,\alpha)$.
Applying \cite[Theorem 8.1, p.86]{olver1974asymptotics} to 
our case, with $x=m,\ p(t)=f(t),\ q(t)=g(t),\ s=j,\ \mu=1$ and $\lambda=1/q$,
we establish the existence of real constants $\gamma_j,\ j=0,1,2\ldots$,
such that as $m \to \infty$,
\begin{align*}
\label{eq:jsum}
J(m,\alpha)
& \sim \frac{m^{1/q}}{q} e^{-mf(0)} \sum_{j=0}^\infty \Gamma \left(j+\frac{1}{q}\right)
\frac{\gamma_j}{m^{j+1/q}}
=\frac{1}{q}\sum_{j=0}^\infty \Gamma \left(j+\frac{1}{q}\right)
\frac{\gamma_j}{m^{j}},
\end{align*}
where $\Gamma$ is the gamma function.

The coefficients $\gamma_0$ and $\gamma_1$ are given in \cite[(8.07), p.86]{olver1974asymptotics} as follows: 
by \eqref{eq:g_j}, \eqref{eq:a_j} and \eqref{eq:f_j} we obtain
\begin{align*}
\gamma_0=\frac{g_0}{\mu f_0^{\lambda/\mu}}=(p(q+1))^{1/q},
\end{align*}
and 
\begin{align*}
\gamma_1
&=\left(\frac{g_1}{\mu}-\frac{(\lambda+1)f_1g_0}{\mu^2 f_0} \right)
\frac{1}{f_0^{(\lambda+1)/\mu}}=-\left(\frac{1}{q}+1\right) (p(q+1))^{2+1/q} \left(-a_2+\frac{a_1^2}{2}\right)\\
&=-\frac{(p(q+1))^{1/q}(q+1)(pq^2+2pq-3q^2+p-2q)}{2q(2q+1)}.
\end{align*}
This completes the proof.
\end{proof}
Next we present an immediate corollary of Theorem \ref{asyexp} on the limit of the integral $I_{p,q}(m)$ as $m \to \infty$. Alternatively and independent of this, we give a self-contained proof for 
the corollary.
\begin{cor}\label{pqBall}
Let $p,\ q \in (1,\infty)$. Then,
$$\lim_{m \to \infty} I_{p,q}(m)=\frac{1}{q}\Gamma \left(\frac{1}{q}\right)(p(q+1))^{1/q}.$$
\end{cor}

\begin{proof}
Let $\ep \in (0,1)$ be any number and $\alpha:=\sin_{p,q}^{-1}{\ep} \in \left(0,
\frac{\pi_{p,q}}{2}\right)$.
For this $\alpha$, we
define $J(m, \alpha)$ as \eqref{Jm}. Changing the variable in the integral $J(m,\alpha)$ to $y=\sin_{p,q}{x}$ we get
$$J(m,\alpha)=m^{1/q}\int_0^\ep \left(\frac{y}{\sin_{p,q}^{-1}{y}}\right)^m
(1-y^q)^{-1/p}\,dy.$$
From the result of Bhayo and Vuorinen \cite[Theorem 1.1 (1)]{bhayo2012generalized}, 
we have
\begin{equation*}
(1-y^q)^{1/(p(q+1))}<\frac{y}{\sin_{p,q}^{-1}{y}}<\left(1+\frac{y^q}{p(q+1)}\right)^{-1},
\quad y \in (0,\ep).
\end{equation*}
Thus,
\begin{equation}
\label{eq:Jineq}
\int_0^\ep (1-y^q)^{m/(p(q+1))-1/p}\,dy
<\frac{J(m,\alpha)}{m^{1/q}}
<\int_0^\ep \left(1+\frac{y^q}{p(q+1)}\right)^{-m}(1-y^q)^{-1/p}\,dy.
\end{equation}

We denote by $L(m,\ep)$ the left-hand side of \eqref{eq:Jineq}
and by $R(m,\ep)$ the right-hand side of \eqref{eq:Jineq}.

$L(m,\ep)$ is estimated as follows.
\begin{align*} 
L(m,\ep)
=\int_0^1-\int_\ep^1 (1-y^q)^{m/(p(q+1))-1/p}\,dy
=:L_1-L_2.
\end{align*}
We obtain
\begin{align}
L_1
&=\frac{1}{q}\int_0^1 z^{1/q-1}(1-z)^{m/(p(q+1))-1/p}\,dz =\frac{1}{q}B\left(\frac{1}{q},\frac{m}{p(q+1)}-\frac{1}{p}+1\right) \notag \\
&=\frac{1}{q}\Gamma\left(\frac{1}{q}\right)
\frac{\Gamma(m/(p(q+1))-1/p+1)}{\Gamma(m/(p(q+1))-1/p+1+1/q)}\notag\\
& \sim \frac{1}{q}\Gamma\left(\frac{1}{q}\right)(p(q+1))^{1/q}m^{-1/q}. \label{eq:J_1}
\end{align}
The last equivalence is due to the fact that
\begin{equation*}
\frac{\Gamma(m+a)}{\Gamma(m+b)} \sim m^{a-b} \quad \mbox{as}\ m \to \infty,
\end{equation*}
which follows from Stirling's formula; see also \cite[Problem 2, p.45]{henriciapplied}.
Moreover, for $m$ large enough,
\begin{equation}
\label{eq:J_2}
L_2<(1-\ep^q)^{m/(p(q+1))-1/p}(1-\ep)=o(m^{-1/q}).
\end{equation}
Then from \eqref{eq:J_1} and \eqref{eq:J_2},
\begin{align}
\liminf_{m \to \infty} m^{1/q}L(m,\ep)
=\liminf_{m \to \infty}m^{1/q}(L_1-L_2)
 \geq \frac{1}{q}\Gamma\left(\frac{1}{q}\right)(p(q+1))^{1/q}. \label{eq:L}
\end{align}

Similarly, $R(m,\ep)$ is estimated as follows.
We obtain
\begin{align*}
R(m,\ep)<(1-\ep^q)^{-1/p}\int_0^\infty \left(1+\frac{y^q}{p(q+1)}\right)^{-m}\,dy.
\end{align*}
Letting $y^q=p(q+1)(1-z)/z$,  we see that for $m$ large enough, 
the integral of the right-hand side can
be rewritten as
\begin{align*}
\frac{1}{q}(p(q+1))^{1/q}B\left(m-\frac{1}{q},\frac{1}{q}\right)
\sim \frac{1}{q}\Gamma\left(\frac{1}{q}\right)(p(q+1))^{1/q}m^{-1/q}.
\end{align*}
Then
\begin{equation}
\label{eq:R}
\limsup_{m \to \infty} m^{1/q}R(m,\ep) \leq (1-\ep^q)^{-1/p}
\frac{1}{q}\Gamma\left(\frac{1}{q}\right)(p(q+1))^{1/q}.
\end{equation}

Applying Lemma \ref{lem:Ipqm}, \eqref{eq:L} and \eqref{eq:R} to \eqref{eq:Jineq}, we have
$$\frac{1}{q}\Gamma\left(\frac{1}{q}\right)(p(q+1))^{1/q}
\leq \lim_{m\to \infty}I_{p,q}(m) \leq 
(1-\ep^q)^{-1/p}\frac{1}{q}\Gamma\left(\frac{1}{q}\right)
(p(q+1))^{1/q}.$$
These inequalities hold for any $\ep \in (0,1)$. 
Note that both sides of the first inequality are independent of $\ep$. 
Therefore, as $\ep \to 0^+$, we conclude
$$\lim_{m \to \infty} I_{p,q}(m)=\frac{1}{q}\Gamma \left(\frac{1}{q}\right)(p(q+1))^{1/q},$$
and the proof is complete.
\end{proof}

We finish by observing that the asymptote to the function $I_{p,q}(m)$ as $m \to \infty$ is $\frac{1}{q}\Gamma\left(\frac{1}{q}\right)(p(q+1))^{1/q}$ for all $p, q \in(1,\infty)$. While the behaviour of the integral $I_{p,q}(m)$ is not fully understood yet for the values of $m$ in the interval $(1, \infty)$, one can establish some knowledge about how the function $I_{p,q}(m)$ is approaching its asymptote as $m \to \infty$.

Based on Theorem \ref{asyexp}, for sufficiently large $m$ we may regard $I_{p,q}(m)$ as
\[
\widetilde{I}_{p,q}(m) := \frac{1}{q} \Gamma \left(\frac{1}{q}\right) (p(q+1))^{1/q}
\left[ 1-\frac{(q+1)(pq^2+2pq-3q^2+p-2q)}{2q^2(2q+1)}\frac{1}{m}\right].
\] 
Now choose $p,\ q \in (1,\infty)$ so that $p<\frac{q(3q+2)}{(q+1)^2}$, then the function $\widetilde I_{p,q}(m)$ becomes convex and decreasing for sufficiently large $m$ as opposed to the classical case when $p=q=2$ (where the integral $\widetilde I_{2,2}(m)$ is concave and increasing when approaching its asymptote $\sqrt{\frac{3\pi}{2}}$ as $m \to \infty$).\\ 

We end this remark by leaving the following as open questions:
\begin{enumerate}[(a)]
\item Explicit computation of the coefficient $\gamma_2$ in the asymptotic expansion in Theorem \ref{asyexp} (cf. $c_2=-\sqrt{\frac{3\pi}{2}}\frac{13}{1120}$ in \eqref{BII}).
\item Behaviour of the integral $I_{p,q}(m)$ as a function of $m \in (1, \infty)$. In particular, the supremum of $I_{p,q}(m)$ on $[2,\infty)$ gives
a generalisation of Ball's integral inequality \eqref{eq:ball}.
\item Inspired by the explicit formula (which can be found in \cite[Exercise 22, p.471]{l1926introduction} attributed to Wolstenholme),
\[
\int_0^\infty \left(\frac{\sin x}{x}\right)^n dx=\frac{1}{(n-1)!}\ \frac{\pi}{2^n} \ \sum_{k=0}^{\lfloor n/2 \rfloor}(-1)^k {n \choose k} (n-2k)^{n-1}
\]
for integers $n \geq 1$, we refer to Remark \ref{remark1} wondering whether it is possible to obtain a similar expression for integrals of $\left(\sinc_{p,q}\right)^r$ when $r=n \in\mathbb{N}$.
\end{enumerate} 
\section{Acknowledgment}
The work of ST was supported by JSPS KAKENHI Grant Number 17K05336.
HM thanks Prof. D. E. Edmunds for his invaluable and generous comments, and endorses her former MSc student Mr. Andrew Pritchard for obtaining (independent of this work) an argument analogous to that of Corollary \ref{pqBall}. 
%
%
%

\end{document}